\documentclass[graybox]{svmult}

\usepackage{type1cm}        
\usepackage{makeidx}         
\usepackage{graphicx}        
\usepackage{multicol}        
\usepackage[bottom]{footmisc}
\usepackage[latin1]{inputenc}

\usepackage{newtxtext}       %
\usepackage[varvw]{newtxmath}       

\usepackage[all]{xy}
\usepackage[mathscr]{eucal}

\makeindex             

\newcommand{\Mm}{{\mathscr M}}

\newcommand{\Bb}{{\mathscr B}}

\newcommand{\Xx}{{\mathscr X}}
\newcommand{\Yy}{{\mathscr Y}}

\newcommand{\Hh}{{\mathscr H}}

\newcommand{\Hom}{{\rm Hom}}

\newcommand{\Spec}{{\rm Spec}}

\newcommand{\BBO}{{\Bb un}}
\newcommand{\BBBS}{{\Bb un}_{G, X}}

\newcommand{\hm}{{\Hh \hspace*{-0.1cm}om}}
\newcommand{\mM}{{\Mm \hspace*{-0.05cm}ap}}

\def\Hom{\mathrm{Hom}}

\def\Spec{\mathrm{Spec}}

\def\I{{I}}

\def\dim{\mathrm{dim\,}}

\def\E{{\mathcal E}}

\def\I{{\mathcal I}}

\def\Xx{{\mathscr X}}
\def\Yy{{\mathscr Y}}

\spnewtheorem{lemm}[theorem]{Lemma}{\bf }{}
\spnewtheorem{cor}[theorem]{Corollary}{\bf }{}
\spnewtheorem{prop}[theorem]{Proposition}{\bf }{}
\spnewtheorem{defn}[theorem]{Definition}{\bf }{}

\begin{document}

\title*{Rational Homotopy and Hodge Theory of Moduli Stacks of Principal $G$-Bundles}
\titlerunning{Rational Homotopy and Hodge Theory of Moduli Stacks of Principal Bundles}

\author{Pedro L. del Angel Rodriguez and Frank Neumann}
\institute{Pedro L. del Angel Rodri­guez  \at Centro de Investigaci\'on en Matem\'aticas (CIMAT)\\
Jalisco S/N, Col. Valenciana\\
36023 Guanajuato, Gto, M\'exico\\ \email{luis@cimat.mx}
\and Frank Neumann \at Dipartimento di Matematica 'Felice Casorati'\\
Universit\`a di Pavia\\
via Ferrata 5, 27100 Pavia, Italy\\ \email{frank.neumann@unipv.it}}
%
%
\maketitle

\abstract*{For a semisimple complex algebraic group $G$ we determine the rational cohomology and the Hodge-Tate structure of the moduli stack $\BBBS$ of principal $G$-bundles over a connected smooth complex projective variety $X$ of special type using the homotopy theory of the underlying topological stack. }

\abstract{For a semisimple complex algebraic group $G$ we determine the rational cohomology and the Hodge-Tate structure of the moduli stack $\BBBS$ of principal $G$-bundles over a connected smooth complex projective variety $X$ of special type using the homotopy theory of the underlying topological stack.}

\section*{Introduction}

\noindent The moduli stack of principal $G$-bundles over a complex algebraic curve plays a fundamental role in algebraic geometry and mathematical physics. In their seminal paper \cite{atiyah} of 1983, Atiyah and Bott, applying methods from Morse theory to the Yangs-Mills functional over a compact Riemann surface (or a projective smooth algebraic curve) $C$, compute the cohomology of the moduli spaces of principal $U(n)$-bundles over $C$ and outlined the essential steps towards a similar computation for general semisimple Lie groups. Later, in 2006, Dhillon \cite{dhillon} computed the Betti numbers as well as the Hodge numbers of the moduli space of stable vector bundles of rank $r$ and degree $d$ over a smooth projective curve $C$ and, in 2008, Arapura and Dhillon \cite{arapura} went a step further to determine the relative Andr\'e motive as well as the rational cohomology of the moduli stack of $G$-bundles over the universal curve. 

Inspired by these previous works, we will show how, for a semisimple complex algebraic group $G$, we can determine the Hodge-Tate structure of the moduli stack $\BBBS$ of principal $G$-bundles over a connected smooth complex projective variety $X$ of a special type depending on the underlying topological CW-structure. 

A key ingredient of our work will be the determination of the rational homotopy type of the moduli stack $\BBBS$ and its underlying topological mapping stack using the construction of homotopy types for topological stacks as developed by Noohi \cite{No, No1, No2} and the classical work by Thom \cite{thom} on the homotopy and cohomology of mapping spaces.

\section{Moduli stacks of principal bundles}

\noindent In this section we will briefly recall some aspects of the general theory of algebraic principal bundles over algebraic varieties and their moduli stacks. More details can be found, for example, in \cite{Sorger} or \cite{gaitsgorylurie}.

\begin{defn} Let $S$ be a scheme and $G$ a group scheme over $S$. For every scheme $X$ over $S$, let 
$G_X=G\times_S X$ be the associated group scheme over $X$. A {\em principal $G$-bundle over $X$} is a scheme $P$ over $X$ equipped with an action
$$G_X\times_X P\cong G\times_SP\rightarrow P$$
of $G_X$ in the category of schemes over $X$, which is locally trivial in the following sense: there exists a faithful flat covering morphism
$U\rightarrow X$ and a $G_X$-equivariant isomorphism
$$U\times_X P\cong U\times_X G_X\cong U\times_S G.$$
\end{defn}

Given a morphism of schemes $f: X'\rightarrow X$ over $X$ and a principal $G$-bundle $P$ over $X$, then the fibre product $X'\times_X P$ exists and is a principal $G$-bundle over $X'$. We call $X'\times_X P$ the {\em pullback of $P$ along $f$}, and denote it by $f^*P$. 

If $G$ is a smooth group scheme over $S$, any principal $G$-bundle $P$ over a scheme $X$ over the base $S$ is also smooth as a scheme over $X$ and can actually be trivialised by an \'etale covering morphism $U\rightarrow X$.

Let us recall now the general definition of the moduli stack of principal bundles over an algebraic stack as a Hom-stack (compare \cite{LMB, hallrydh}).

\begin{defn}\label{Bun} Let $G$ be a group scheme over a base scheme $S$ and $\Xx$ be an algebraic stack over the category $Sch/S$ of schemes over $S$. The moduli stack of principal $G$-bundles over $\Xx$ is defined as the Hom-stack
$$\BBO_{G, \Xx}= \hm_S (\Xx, \Bb G)$$
where $\Bb G$ is the classifying stack of $G$.
\end{defn}

In the particular situation that $\Xx$ is a projective flat scheme $X$ over a Noetherian base scheme $S$ and $G$ is a closed subgroup scheme of $GL_n$ and the fppf-quotient scheme $GL_n/G$ is quasi-projective over $S$ it follows that the moduli stack $\BBO_{G, X}$ is actually an algebraic stack which is locally of finite type (see \cite{hallrydh}).

In this article we will always work over the field ${\mathbb C}$ of complex numbers and our base scheme will be $S=\Spec(\mathbb C)$. We also assume that $X$ is a connected smooth complex projective variety and $G$ will be a semisimple complex affine algebraic group. The moduli stack $\BBO_{G, X}$ can then also be described via its groupoid of sections as follows: For a given scheme $U$ over $\Spec(\mathbb C)$, the groupoid $\BBBS(U)$ of $U$-valued points of $\BBBS$ is the groupoid of principal $G$-bundles over $X\times_{\Spec(\mathbb C)} U$ and their isomorphisms. Then we have the following (see \cite{LMB, behrend1, behrend2})

\begin{theorem}
The moduli stack  $\BBBS$ is a smooth Artin algebraic stack, which is locally of finite type.  \hfill{ }$\square$
\end{theorem}

\section{Homotopy types of algebraic stacks}

\noindent We recollect here some facts about homotopy types of algebraic and topological stacks following Noohi \cite{No, No1, No2} (compare also \cite{FN}).

In order to analyze homotopy types of complex algebraic stacks we need to understand homotopy types in the underlying topological situation.

Let $\Xx$ be an algebraic stack over the catgegory $\bf{Sch}/\mathbb C$ of schemes over $\mathbb C$.  If $\Xx$ is locally of finite type, then there exists an associated topological stack $\Xx^{top}$ over the category $\bf{Top}$ of topological spaces.

A {\em topological stack} is basically a stack $\Xx$ over $\bf{Top}$ which admits an atlas, i. e. a representable morphism $p: X\rightarrow \Xx$ from a topological space $X$ having local sections (see \cite{No} for the general theory of stacks over the site $\bf{Top}$).

Noohi \cite[Prop.~20.2]{No} constructs a well-defined topologification functor $(-)^{top}:\mathfrak{AlgSt}/\mathbb C\rightarrow \mathfrak{TopSt}$ from the $2$-category of algebraic stacks over $\mathbb C$ to the $2$-category of topological stacks. It sends Artin algebraic stacks $\Xx$ , i.e. algebraic stacks with a smooth atlas to topological stacks $\Xx^{top}$ whose atlas is a local topological fibration. These topological stacks can be considered as the analogs of Artin algebraic stacks in the topological framework.
This topologification functor also factors naturally through the $2$-category of complex analytic stacks $\mathfrak{AnSt}$ and it basically extends the complex points functor from the category $\bf{Sch}/\mathbb C$ of schemes over $\mathbb C$ to the category $\bf {AnSp}$ of complex analytic spaces. Here a {\em complex analytic stack} $\Xx$ is a stack over the site $\bf {AnSp}$ of complex analytic spaces which admits an atlas,~i.~e.~a complex analytic smooth surjection $X\rightarrow \Xx$ from a complex analytic space $X$ (see also \cite{FN}).

Let us now define the homotopy type of a topological stack (see \cite{No}, \cite{No2}):

\begin{defn}
A representable morphism $f: \Xx\rightarrow \Yy$ between topological stacks is called a {\em universal weak equivalence} if for any morphism $Y\rightarrow \Yy$ from a topological space $Y$, the left vertical morphism in the diagram
\[
\xy \xymatrix{Y\times_{\Yy} \Xx \ar[r] \ar[d]  & \Xx\ar[d]^f \\
Y\ar[r]& \Yy}
\endxy
\]
is a weak homotopy equivalence of topological spaces.
A {\it homotopy type} for a topological stack $\Xx$ is a CW-complex $\{\Xx\}$ together with a morphism $\eta: \{\Xx\}\rightarrow \Xx$ which is a universal weak equivalence.
\end{defn}

Every universal weak equivalence is automatically a representable morphism. We could have used also the simplicial space associated to an atlas to define the homotopy type of a topological stack (see \cite{FN}).
Let $\Xx$ be a topological stack with atlas $X\rightarrow \Xx$. The associated simplicial topological space $X_{\bullet}$ is given by the iterated fiber products of the atlas with $X_n=X\times_{\Xx}\ldots \times_{\Xx}X$ having $(n+1)$ factors. $X_{\bullet}$ is simply the nerve of the topological groupoid $[X\times_{\Xx} X\rightrightarrows X]$ associated to the topological stack $\Xx$. We  have the following relation between the homotopy type of a topological stack $\Xx$ and the homotopy type of the (thick) geometric realization of $X_{\bullet}$ (see \cite{No}):

\begin{prop}
Let $\Xx$ be a topological stack and $X\rightarrow \Xx$ be an atlas. There is a universal weak equivalence
$||X_{\bullet}||\rightarrow \Xx$, where $||X_{\bullet}||$ is the (thick) geometric realisation of the simplicial topological space $X_{\bullet}.$
\end{prop}

The topological space $||X_{\bullet}||$ is in general {\em not} a CW-complex, but the (thick) geometric realization of the singular simplicial set $Sing_{\bullet}(||X_{\bullet}||)$ is always a CW-complex. In fact, the evaluation map
$|Sing_{\bullet}(||X_{\bullet}||)|\rightarrow ||X_{\bullet}||$ is then a Serre fibration and a weak homotopy equivalence. Therefore the composition $|Sing_{\bullet}(||X_{\bullet}||)|\rightarrow \Xx$ of the evaluation map and the universal weak equivalence above is again a universal weak equivalence and $|Sing_{\bullet}(||X_{\bullet}||)|$ is a homotopy type for $\Xx$ which proves:

\begin{prop}
Any topological stack $\Xx$ has a homotopy type $\{\Xx\}$ which is given as a CW-complex.
\end{prop}

Any CW-complex $X$ can be considered as a topological stack and therefore itself has a homotopy type $\{X\}$ which is simply given as the identity map $X\rightarrow X$, which is a universal weak equivalence and it follows immediately that $\{X\}\simeq X$.

For the action of a topological group $G$ on a topological space $X$ we have the {\em quotient stack} $[X/G]$ defined as follows: For every topological space $U$ the groupoid $[X/G](U)$ is the groupoid of all triples $(E, \pi, \mu)$, where $\pi: E\rightarrow U$ is a principal $G$-bundle and $\mu: E\rightarrow X$ is a $G$-equivariant map and the isomorphisms are simply $G$-bundle isomorphism compatible with the respective equivariant maps.

It follows that $[X/G]$ is a topological stack. An atlas $q: X\rightarrow [X/G]$ is given by the triple $(G\times X, pr_X, \rho)$, where $pr_X: G\times X\rightarrow X$ is the trivial bundle and $\rho: G\times X\rightarrow X$ the action. This data defines an object in $[X/G](X)$ and via representability a morphism $q: X\rightarrow [X/G]$, which in fact is a principal $G$-bundle.

It is easy to see that the homotopy type $\{[X/G]\}$ of the quotient stack $[X/G]$ is given by the Borel construction $EG\times_G X$, where the universal weak equivalence is the morphism
$$\eta: EG\times_G X\rightarrow [X/G]$$
given by the projection map $EG\times X\rightarrow X$, which is $G$-equivariant and the quotient map
$EG\times X\rightarrow EG\times_{G}X$ which is a principal $G$-bundle. In the special case that the action is free the quotient stack $[X/G]$ is isomorphic to the quotient space $X/G$. If $X=*$ is a point with a trivial action of a topological group $G$, then it follows that the homotopy type of the {\em classifying stack} $\Bb G=[*/G]$ is homotopy equivalent to the classifying space $BG$, i.~e.~$\{\Bb G\}\simeq BG$.

In order to study homotopy types of moduli stacks of principal bundles over complex algebraic varieties we need to understand the homotopy type of mapping stacks. Giving two topological stacks $\Xx$ and $\Yy$ we have the associated mapping stack $\mM(\Xx, \Yy)$, which is basically the topological version of the algebraic Hom-stack given by the
groupoid of section as
$$\mM(\Xx, \Yy)(U)=\Hom(U\times \Xx, \Yy)$$
with $\Hom(U\times \Xx, \Yy)$ the respective groupoid of stack morphisms (see \cite{BGNX, No1}).
It turns out that if $X$ is a compact topological space, then the mapping stack $\mM(X, \Yy)$ is actually a topological stack \cite[Prop. 5.1]{BGNX}. In this particular case, the homotopy type of the mapping stack $\mM(X, \Yy)$ is represented by the mapping space ${\operatorname Map}(X, Y)$, where $Y$ represents the homotopy type of the topological stack $\Yy$ (see \cite[Cor. 6.5]{No1}).

We can now define the (classical) homotopy type of an algebraic stack over the complex numbers via its underlying topological stack.

\begin{defn}
Let $\Xx$ be an Artin algebraic stack over $\mathbb C$, which is locally of finite type. The {\em (classical) homotopy type} of $\Xx$ is defined as the homotopy type of the underlying topological stack $\Xx^{top}$, i.e. $\{\Xx\}_{cl}=\{\Xx^{top}\}.$
\end{defn}

Now we can use the machinery of homotopy theory for topological stacks to determine the (classical) homotopy type of moduli stacks of principal $G$-bundles over fixed smooth complex projective algebraic varieties. 

\begin{theorem}
Let $G$ be a semisimple affine complex algebraic group and let $X$ be a connected smooth complex projective algebraic variety of special type, i.e. the homotopy type of the moduli stack $\BBO_{G, X}$ is given as the homotopy type of the topological mapping stack $\mM(X,\mathscr BG)$. Then the homotopy type of the moduli stack $\BBO_{G, X}$ is given as $\{ \BBO_{G, X} \}_{cl}\simeq {\operatorname Map} (X, BG)$ where $BG$ is the classifying space of $G$. 
\end{theorem}

\begin{proof}
From Definition \ref{Bun} and the condition for $X$ being of special type, we see that the homotopy type of $\BBO_{G,X}$ is given as the homotopy type of the topological mapping stack
$$\{\BBO_{G,X}\}_{cl}=\{\mM(X,\mathscr BG)\}.$$
For any representation of the classifying space $BG$, the universal bundle $EG$ over $BG$ defines an object in the groupoid $\mathscr BG(BG)$ and via representation, a map $BG\to\mathscr BG$, which is a weak universal equivalence, as mentioned above.

Now, for any complex projective algebraic variety (and indeed for any compact CW-complex), any map $X\to BG$ induces by composition a map $X\to\mathscr BG$, therefore we get a natural morphism
\[
Map(X, BG) \to \mM(X,\mathscr BG)
\]
which by Noohi [No 1 Cor. 6.5] is again a weak universal equivalence. Therefore the
homotopy type of the mapping stack $\mathscr Map (X,\mathscr BG)$ is given by the mapping space
$Map (X, BG)$, which has the homotopy type of a CW-complex \cite{milnor}. 
\end{proof}

In the special case, that $X$  is a connected smooth complex projective curve, the condition of being special is fulfilled and the previous statement was obtained before by Atitah-Bott \cite{atiyah} using gauge theory and equivariant Morse theory and later also by Teleman \cite{teleman} using simplicial homotopy theory (compare also \cite{arapura} and \cite{gaitsgorylurie}). 

\section{Rational homotopy of the moduli stack $\BBBS$}

\noindent Let $G$ be a semisimple affine complex algebraic group over $\mathbb C$. If $T\subset G$ is a maximal torus with Weyl group $W=N(T)/T$, and $BG$ denotes the classifying space of $G$, then
\begin{eqnarray} \label{eqn:CohomologyRing}
H^*(BG,{\mathbb Q})\cong H^*(BT,{\mathbb Q})^W\cong {\mathbb Q}[x_1,\dots,x_n]^W,
\end{eqnarray}
where $T\cong {\mathbb G}_m^n$, $BT$ is the classifying space of $T$ and $x_i$ is the first Chern class of the universal line bundle on the $i$-th factor. The right hand side of equation \ref{eqn:CohomologyRing} is a polynomial ring on the elementary $W$-invariant polynomials on the classes $x_i$. These are Chern classes of the universal bundle $EG$ over $BG$ of degree $2n_i$ and $H^*(BG,{\mathbb Q})$ is a direct sum of Tate-Hodge structures.  Let $V=\oplus_i V_i$ be the span of these cohomology classes. Since $G$ is semisimple these numbers are always $\ge 2$ and define a map from the classifying space of $G$
\[
BG\to\prod_i K({\mathbb Z}, 2n_i)
\]
to a product of  Eilenberg-MacLane spaces $K({\mathbb Z},2n_i)$, which induces a \emph{rational} homotopy equivalence. 

Let us recall, that given a connected finite dimensional CW-complex $X$, one has for the homotopy groups of the mapping space from $X$ into an Eilenberg-MacLane space $K(A, n)$ the following description in terms of cohomology groups (see \cite{AgGiPr, Sw})

\begin{eqnarray} \label{eqn:HomotopyGroupsMapSpace}
\begin{matrix}
\pi_k \left( {\operatorname Map}(X,K(A,n))\right) & \cong & \pi_0 \Omega^k {\operatorname Map}(X, K(A,n)) \hfill{ }\cr
 & \cong & \pi_0 {\operatorname Map}(X, \Omega^k K(A,n)) \hfill{ } \cr
 & \cong & \pi_0 {\operatorname Map}(X, K(A,n-k)) \hfill{ } \cr
 & \cong & H^{n-k}(X, A). \hfill{ }
\end{matrix}
\end{eqnarray}
Moreover, for any connected and finite dimensional CW-complex (and in particular any connected complex manifold) we have the following theorem due to Thom \cite{thom}

\begin{theorem}[Thom] \label{thm:Thom}
Let $X$ be a connected finite dimensional CW-complex, then, up to homotopy equivalence, we have
\[
{\operatorname Map}(X, K(A,n)) \hskip0.2cm{}\simeq \prod_{\begin{matrix} 0\le q\le \dim_{\mathbb R}(X)\cr 0\le n-q\end{matrix}} K(H^q(X,A); n-q) \hskip1cm{}
\]
\hfill{ }$\square$
\end{theorem}

Therefore, in particular, for any connected finite dimensional CW-complex $X$, one has a rational homotopy equivalence
\begin{eqnarray} \label{eqn:Thom}
{\operatorname Map}(X, BG) \simeq_{\mathbb Q} \prod_i \prod_{0\le q\le \dim_{\mathbb R} X} K(H^q(X,{\mathbb Z}); 2n_i-q)
\end{eqnarray}
with the convention $K(A;r)=\{*\}$ for $r<0$.\\\\
\noindent Let us look at some instructive examples to illustrate this.

\begin{example}{Examples} 
\begin{itemize}
\item[1)]\label{ex:curve}
If $C$ is a complex algebraic curve of genus $g$, then $H^q(X,{\mathbb Z})$ is either ${\mathbb Z}$ (for $q=0, 2$) or ${\mathbb Z}^g$ (for $q=1$), and therefore
\[
\operatorname{Map} (C, BG) \simeq_{\mathbb Q} \prod_{i} \left(K({\mathbb Z},2n_i)\times K({\mathbb Z}^g, 2n_i-1)\times K({\mathbb Z}, 2n_i-2)\right)
\]
\item[2)]\label{ex:S1} If $X=S^1$, we get
\[
\operatorname{Map}(S^1, BG) \simeq_{\mathbb Q} \prod_{i} \left(K({\mathbb Z},2n_i)\times K({\mathbb Z}, 2n_i-1)\right)
\]
\item[3)] \label{ex:Pk} For $X=\mathbb P^k$, we get
\[
\operatorname{Map}({\mathbb P}^k, BG) \simeq_{\mathbb Q} \prod_i \prod_{0\le q\le k} K({\mathbb Z}; 2n_i-2q)
\]
again with the convention that $K({\mathbb Z}, r)=\{*\}$ for $r<0$.
\end{itemize}
\end{example}
\noindent A more interesting example is given by smooth projective hypersurfaces

\begin{lemm} \label{lem:hypersurface}
Assume $X\subset{\mathbb P}^{k+1}$ is a smooth hypersurface, then the Poincar\'e series of ${\operatorname Map}(X, BG)$ is of the form
\[
P_t\left({\operatorname Map}(X, BG)\right) = {\displaystyle \prod_i \frac{1}{(1-t^{2n_i-2m})^d}\times\prod_{0\le q<m}\frac{1}{1-t^{2n_i-2q}} }
\]
if $k= 2m$ is even and $\dim H^k(X,{\mathbb Q})=d$, or
\[
P_t\left({\operatorname Map}(X, BG)\right) = {\displaystyle \prod_i (1+t^{2n_i-2m-1})^d\times\prod_{0\le q<m}\frac{1}{1-t^{2n_i-2q}} }
\]
if $k= 2m+1$ is odd and $\dim H^k(X,{\mathbb Q})=d$. 
\end{lemm}

\begin{proof} By Thom's theorem we have
\[
{\operatorname Map}(X, BG) \simeq_{\mathbb Q} \prod_i \prod_{0\le q\le \dim_{\mathbb R} X} K(H^q(X,{\mathbb Z}); 2n_i-q).
\]
By the weak Lefschetz theorem we have
$$H^j(X,{\mathbb Q})\cong H^j({\mathbb P}^{k+1},{\mathbb Q})$$ for all $0\le j < k$. Therefore by Poincar\'e duality, $H^j(X,{\mathbb Q})$ is either $0$ or ${\mathbb Q}$ for $j\ne k$, depending on wether $j$ is odd or even. And so the above rational homotopy decomposition becomes
\[
{\operatorname Map}(X, BG) \simeq_{\mathbb Q} \prod_i \left(\prod_{2q \ne k} K({\mathbb Z}; 2n_i-2q)\right)\times K(H^k(X, {\mathbb Z});2n_i - k).
\]
Since the Poincar\'e series are multiplicative and we have
\[
\begin{matrix}
P_t\left( K({\mathbb Z}, 2r)\right) \hfill{ } & = &{\displaystyle\frac{1}{1-t^{2r}}}\cr
\cr
P_t\left( K({\mathbb Z}, 2r-1)\right) & = & \;\;\; 1+t^{2r-1},
\end{matrix}
\] 
we get therefore
\[
P_t\left({\operatorname Map}(X, BG)\right) = {\displaystyle \prod_i \frac{1}{(1-t^{2n_i-2m})^d}\times \prod_{0\le q<m}\frac{1}{1-t^{2n_i-2q}} }
\]
if $k= 2m$ and $\dim H^k(X,{\mathbb Q})=d$, or
\[
P_t\left({\operatorname Map}(X, BG)\right) = {\displaystyle  \prod_i (1+t^{2n_i-2m-1})^d\times \prod_{0\le q<m}\frac{1}{1-t^{2n_i-2q}} }
\]
if $k= 2m+1$ and $\dim H^k(X,{\mathbb Q})=d$, as claimed. 
\end{proof}

As in Atiyah-Bott \cite{atiyah} (see also \cite{teleman, arapura}), we can identify $G$ with the based loop space 
$\Omega BG={\operatorname Map}^*(S^1,BG)$ and up to rational homotopy equivalence get a splitting
\[
G\hskip0.2cm{ }\simeq \hskip0.2cm{ } {\operatorname Map}^*(S^1, BG) \hskip0.2cm{ } \simeq_{\mathbb Q} \prod_i K({\mathbb Z}, 2n_i-1).
\]
The universal principal $G$-bundle over $S^1\times {\operatorname Map}^*(S^1,BG)$ induces an evaluation map
\[
S^1\times G\cong S^1\times  {\operatorname Map}^*(S^1,BG)\to BG.
\] 
And so, the pullback map followed by the projection and integration along $S^1$ produces a commutative diagram of the form
{\small
\[
\xymatrix{
H^*(BG,{\mathbb Q})\ar[r] \ar[rrdd] & H^*(S^1\times G)\ar[r]^-{\cong} &  H^0(S^1,{\mathbb Q})\otimes H^*(G,{\mathbb Q}) \oplus H^1(S^1,{\mathbb Q})\otimes H^{*-1}(G,{\mathbb Q})\ar[d] \cr
& & \text{ } \hfill{} H^1(S^1,{\mathbb Q})\otimes H^{*-1}(G,{\mathbb Q})\ar[d]^{\int_{S^1}} \cr
& & H^{*-1}(G,{\mathbb Q}).
}
\]
}
\\
As before, let $V$ be the the span of the Chern classes of the universal principal $G$-bundle $EG$. The image of $V$ under the above composition can then be identified with $V[1]$ where, as usual, $V[i]_k=V_{k+i}$ and it is not difficult to see that we get an isomorphism $H^*(G,{\mathbb Q}) \cong \Lambda\, V[1]$.

\section{Cohomology and Hodge theory of the moduli stack $\BBBS$}

\noindent Assume that $X$ is given as a connected CW-complex of real dimension $n$ and $G$ a semisimple affine algebraic group over $\mathbb C$. Then the previous Lemma \ref{lem:hypersurface} can be generalized as follows

\begin{prop} Let $X$ be a connected CW-complex of real dimension $n$ and $G$ be a semisimple complex algebraic group. 
If $\dim H^j(X,{\mathbb Q})=d_j$ for $0\le j\le n$, then \ \\
the Poincar\'e polynomial of the mapping space ${\operatorname Map}(X, BG)$ is
\[
{\displaystyle \frac{\prod_{q\text{ odd}}\prod_i (1+t^{2n_i-q})^{d_q}}{\prod_{q\text{ even}}\prod_i (1-t^{2n_i-q})^{d_q} } }
\]
where the second product in the numerator and in the denominator involve only those $i$'s satisfying $0\le 2n_i-q$ and $0\le q\le n$. \hfill{ }$\square$
\end{prop}

\begin{proof} By Thom's theorem and according to equation (\ref{eqn:Thom}) we have
\begin{align} \label{eqn:Thom2}
{\operatorname Map}(X, BG) &\simeq_{\mathbb Q} \prod_i \prod_{0\le q\le n} K(H^q(X,{\mathbb Z}); 2n_i-q) \cr \cr
& \simeq_{\mathbb Q} \prod_i \prod_{0\le q\le n} K({\mathbb Z}^{d_q}; 2n_i-q) \hfill{ } \cr\cr
& \simeq_{\mathbb Q}  \prod_{0\le q\le n} \prod_i K({\mathbb Z}; 2n_i-q)^{d_q} \hfill{ }
\end{align}
up to rational weak equivalence. So the assertion follows at once, since the Poincar\'e polynomials are multiplicative and we have
\[
\begin{matrix}
P_t\left( K({\mathbb Z}, 2r)\right) \hfill{ } & = &{\displaystyle \frac{1}{1-t^{2r}}}\cr
\cr
P_t\left( K({\mathbb Z}, 2r-1)\right) & = & \;\;\; 1+t^{2r-1}.
\end{matrix}
\]
\end{proof}

Since rationally $BG\simeq_{\mathbb Q} \prod_i K({\mathbb Z},2n_i)$, equation (\ref{eqn:Thom2}) can be interpreted as a rational homotopy splitting of the form
\begin{align}
{\operatorname Map}(X, BG)& \simeq_{\mathbb Q} BG\times (\Omega BG)^{d_1}\times (\Omega^2 BG)^{d_2}\times\dots\times (\Omega^{n} BG)^{d_{n}} \cr
& \simeq_{\mathbb Q} BG\times \prod_{j=1}^{n} (\Omega^{2j} BG)^{d_{2j}}\times \prod_{j=1}^n (\Omega^{2j-1} BG)^{d_{2j-1}} 
\end{align} 
As we have seen, any connected component $\Omega_0^j BG$ of $\Omega^j BG$ can be identified with the based space
\[
\Omega_0^j BG = {\operatorname Map}^*(S^j, BG).
\]
As before, the universal principal $G$-bundle over $S^j\times Map^*(S^j,BG)$ induces an evaluation map
\[
S^j\times \Omega_0^j BG \cong S^j\times {\operatorname Map}^*(S^j,BG) \longrightarrow BG.
\]
Therefore the pullback map followed by the projection and integration along $S^j$ produces
\[
\xymatrix{
H^*(BG,{\mathbb Q})\ar[r] \ar[rrdd] & H^*(S^j\times \Omega_0^j BG)\ar[r]^-{\cong} &  \oplus_i H^i(S^j,{\mathbb Q})\otimes H^{*-i}(\Omega_0^j BG,{\mathbb Q})\ar[d] \cr
& & \text{ } \hfill{} H^j(S^j,{\mathbb Q})\otimes H^{*-j}(\Omega_0^j BG,{\mathbb Q})\ar[d]^{\int_{S^j}} \cr
& & H^{*-j}(\Omega_0^j BG,{\mathbb Q})
}
\]
The image of $V$ under this composition can be identified with $V[j]$ and it is not difficult to see that
 \[
H^*(\Omega_0^{2j} BG^0,{\mathbb Q}) \cong \text{Sym}\, V[2j] 
\]
and
\[
H^*(\Omega_0^{2j-1} BG^0,{\mathbb Q})\cong \Lambda \,V[2j-1],
 \]
where $G^0$ denotes the component of the identity of $G$.
Similarly, given a CW-complex $X$, the universal principal $G$-bundle over the product $X\times {\operatorname Map}^*(X, BG)$ induces a classifying map 
\[
X\times {\operatorname Map}^*(X, BG) \longrightarrow BG,
\] 
therefore via pullback, one gets
\[
\xymatrix{
H^*(BG,{\mathbb Q})\ar[r]\ar[rd] &  \oplus_{2j} H^{2j}(X,{\mathbb Q})\otimes H^{*-2j}({\operatorname Map}^*(X, BG),{\mathbb Q}) \ar[d] \cr
 & H^{2j}(X,{\mathbb Q})\otimes H^{*-2j}({\operatorname Map}^*(X, BG),{\mathbb Q})
}
\]
or equivalently,
\[
H^*(BG,{\mathbb Q})\otimes H^{2j}(X,{\mathbb Q})^\vee\to H^{*-2j}({\operatorname Map}^*(X, BG), {\mathbb Q}).
\]
\ \\
The map above induces maps
\begin{align*}
\text{Sym}\, V[2j]\otimes H^{2j}(X,{\mathbb Q})^\vee & \to H^{*-2j}{\operatorname Map}^*(X, BG),{\mathbb Q}) \cr
\Lambda\, V[2j-1]\otimes H^{2j-1}(X,{\mathbb Q})^\vee & \to H^{*-2j+1}({\operatorname Map}^*(X, BG),{\mathbb Q})
\end{align*}
which fit together to produce a map
{\small
\[
\Psi:\text{Sym}\, V \otimes\left(\bigotimes_{j=1}^n (\text{Sym}\, V[2j])^{\otimes d_{2j}} \right)
\otimes \left(\bigotimes_{j=1}^n (\Lambda V[2j-1])^{\otimes d_{2j-1}} \right)\to H^*({\operatorname Map}^*(X, BG), {\mathbb Q}).
\]}
\\

\noindent We have the following theorem

\begin{theorem} Let $X$ be a connected CW-complex of real dimension $2n$ and let $G$ be a semisimple complex algebraic group. Assume that $X$ has no cell of odd dimension and that every even dimensional skeleton $X^{2j}$ of $X$ is connected. If $\dim H^{2j}(X,{\mathbb Q})=d_j$ for $0\le j\le n$, then the cohomology of any connected component 
${\operatorname Map}^0(X, BG)$ of ${\operatorname Map}(X,BG)$ is given, by
\[
{\displaystyle H^*({\operatorname Map}^0(X, BG),{\mathbb Q}) \cong \big(\bigotimes_{1\le j\le n} \left(\text{Sym }V[2j]\right)^{\otimes d_j}\big)}
\]
and the isomorphism is induced by the pullback map from the classifying space. 
\end{theorem}

\begin{proof} 
We proceed by induction on $n$, the assertion being clear for $n=0$.
\\
If $n\ge 1$, the inclusion $X^{2n-2}\hookrightarrow X$, of the $2n-2$ skeleton on $X$, produces a cofibration
\[
X^{2n-2}\longrightarrow X\longrightarrow X/X^{2n-2}\simeq \bigvee S^{2n}
\]
which produces a fibration of base point preserving mapping spaces
\begin{eqnarray*}\label{eq:fibration}
\prod_{d_n}\Omega_0^{2n}\sim {\operatorname Map}^*(\bigvee S^{2n}, BG)\longrightarrow {\operatorname Map}^*(X, BG)\longrightarrow {\operatorname Map}^*(X^{2n-2}, BG)
\end{eqnarray*}

\noindent The above fibration, together with the fibration
\[
{\operatorname Map}^*(X, BG)\longrightarrow {\operatorname Map}(X,BG)\longrightarrow BG
\]
yields a trigraded spectral sequence of the following form, where all cohomology groups are rational cohomology groups
{\small \[
E_2^{pqr}=H^p(BG)\otimes H^q\left({\operatorname Map}^*(X^{2n-2}, BG)\right)\otimes H^r\left(\prod_{d_n}\Omega_0^{2n}\right)\Rightarrow H^{p+q+r}({\operatorname Map}(X, BG))
\]}
Observe that the sum of terms on the left is just a sum of $|\pi_1(G)|$ copies of the domain of $\Psi$ (under the current hypothesis, there is no odd cohomology), so induction on $n$ and the equality of Poincare series forces that the spectral sequence collapses and we have $E_2 = E_\infty$ and thus the map $\Psi$ is an isomorphism. 
\end{proof}

Remember that any compact stratified set in the sense of Thom, and therefore any compact complex manifold (regarded as a differentiable manifold), can be triangulated (see for example \cite{whitehead} or \cite{johnson}). Therefore we get the following corollary on the rational cohomology of the moduli stack of principal $G$-bundles over a smooth complex projective variety

\begin{cor}
Let $X$ be a connected smooth complex projective variety of complex dimension $n$ of special type and let $G$ be a semisimple complex algebraic group. If there are no odd dimensional cells on the cellular decomposition of $X$ as a CW-complex (as is the case for $H/P$ for any semisimple group $H$ and any parabolic subgroup $P\subset H$),
and $\dim H^{2j}(X,{\mathbb Q})=d_j$ for $0\le j\le n$, then the rational cohomology of any connected component $\BBBS^0$ of $\BBBS$ is given by
\[
{\displaystyle H^*(\BBBS^0, {\mathbb Q}) \cong \big(\bigotimes_{1\le j\le n} \left(\text{Sym }V[2j]\right)^{\otimes d_j}\big).}
\]
In particular, it is a direct sum of Hodge-Tate structures.
\end{cor}

More generally, without the special condition on the cellular structure we expect the following result

\begin{conjecture}
Let $X$ be a connected smooth complex projective variety of complex dimension $n$ and $G$ be a semisimple complex algebraic group. 
If $\dim H^j(X,{\mathbb Q})=d_j$ for $0\le j\le 2n$, then the rational cohomology of any connected component $\BBBS^0$ of $\BBBS$ is given by
\[
{\displaystyle H^*(\BBBS^0, {\mathbb Q}) \cong \left(\bigotimes_{j=0}^n (\text{Sym}\, V[2j])^{\otimes d_{2j}} \right)
\otimes \left(\bigotimes_{j=1}^n (\Lambda V[2j-1])^{\otimes d_{2j-1}} \right)}
\]
and the isomorphism is induced by the pullback map from the classifying space. In particular, it is a direct sum of Hodge-Tate structures.
\end{conjecture}

\section*{Acknowledgements}
\noindent Both authors would like to thank Prof. Alexander Schmitt for the kind invitation to speak at the Complex Geometry session as part of the 14th ISAAC Congress 2023 at USP in Ribeir\~ao Preito (Brazil). They also like to thank the local organising committee and the ISAAC board for the great hospitality as well as the referee for the useful comments. The first named author also likes to thank the Centro de Investigaci\'on en Matem\'aticas (CIMAT) for financial support. The second named author thanks the Dipartimento di Matematica 'Felice Casorati' of the University of Pavia for financial support.

\end{document}